\documentclass[11pt,a4paper]{amsart}
\usepackage{amssymb,xspace}
\usepackage{amstext}
\usepackage{tikz}
\usetikzlibrary{arrows,decorations.pathmorphing,backgrounds,fit,positioning,shapes.symbols,chains}
\theoremstyle{plain}
\usepackage{amsbsy,amssymb,amsfonts,latexsym}

\usepackage{enumerate}
\usepackage{graphics}

\date{\today}
\title{The multifractal box dimensions of typical measures}
\author{Fr\'ed\'eric Bayart}
\address{
Clermont Universit\'e, Universit\'e Blaise Pascal, Laboratoire de Math\'ematiques, BP 10448, F-63000 CLERMONT-FERRAND -
CNRS, UMR 6620, Laboratoire de Math\'ematiques, F-63177 AUBIERE
}
\email{Frederic.Bayart@math.univ-bpclermont.fr}



\newcommand{\veps}{\varepsilon}

\def\dinfsmall{\underline{\dim}_{*,B}}
\def\dinfbig{\underline{\dim}_B^*}
\def\dinfsmallmu{\dinfsmall(\mu)}
\def\dinfbigmu{\dinfbig(\mu)}
\def\dsupsmall{\overline{\dim}_{*,B}}
\def\dsupbig{\overline{\dim}_B^*}
\def\dsupsmallmu{\dsupsmall(\mu)}
\def\dsupbigmu{\dsupbig(\mu)}
\def\dsuplocmax{\overline{\dim}_{B,\rm loc,max}}
\def\npiq{\mathbf N_\pi^q}
\def\dsuppiq{\overline{\dim}_{\pi, B}^q}
\def\dsuppilocq{\overline{\dim}_{\pi, B,\rm loc}^q}
\def\dinfpiq{\underline{\dim}_{\pi,B}^q}
\def\dinfbigpiq{\underline{\dim}^{*,q}_{\pi, B}}
\def\dinfsmallpiq{\underline{\dim}^{q}_{*,\pi, B}}
\def\dsupbigpiq{\overline{\dim}^{*,q}_{\pi, B}}
\def\dsupsmallpiq{\overline{\dim}^{q}_{*,\pi, B}}
\def\taupiloc{\tau_{\pi,\rm loc}}
\def\taupi{\tau_\pi}
\def\Dinfpi{\underline{D}_\pi}
\def\Dsuppi{\overline{D}_\pi}
\def\Dsuppimax{\overline{D}_{\pi,\rm max}}
\def\Dinfpimin{\underline{D}_{\pi,\rm min}}
\def\taupilocmax{\tau_{\pi,\rm loc,max}}
\def\Dinfpiunif{\underline{D}_{\pi,\rm unif}}
\def\Dsuppiunif{\overline{D}_{\pi,\rm unif}}
\def\Dinfpiunifmin{\underline{D}_{\pi,\rm unif, min}}
\def\Dsuppiunifmax{\overline{D}_{\pi,\rm unif, max}}
\def\ppiq{\mathbf P_\pi^q}

\def\pk{\mathcal P(K)}

\def\dboxinf{\underline{\dim}_B}

\def\dboxsup{\overline{\dim}_B}
\def\dboxsuploc{\overline{\dim}_{B,{\rm loc}}}

\DeclareMathOperator{\supp}{supp}

\newtheorem{theorem}{Theorem}[section]

\newtheorem{lemma}[theorem]{Lemma}

\newtheorem{proposition}[theorem]{Proposition}

\newtheorem{corollary}[theorem]{Corollary}

{\theoremstyle{definition}}
{\theoremstyle{definition}}

{\theoremstyle{definition}}

{\theoremstyle{definition}}

{\theoremstyle{definition}}

{\theoremstyle{definition}}

\newtheorem*{OLSEN}{Theorem B (Olsen)}
\newtheorem*{MYJAK}{Theorem A (Myjak and Rudnicki)}

\begin{document}

\begin{abstract}
We compute the typical (in the sense of Baire's category theorem) multifractal box dimensions of measures on a compact subset of $\mathbb R^d$. Our results are new even in the context
of box dimensions of measures.
\end{abstract}

\maketitle

\section{Introduction}
\subsection{Position of the problem}
The origin of this paper goes back to the work \cite{MR02} of J. Myjak and R. Rudnicki, where they investigate the box dimensions of typical measures.
To state their result, we need to introduce some terminology. Let $K$ be a compact subset of $\mathbb R^d$, and let $\mathcal P(K)$ be the set of Borel probability measures on $K$; we endow $\mathcal P(K)$ with the weak topology.
By a property true for a typical measure
of $\pk$, we mean a property which is satisfied by a dense $G_\delta$ set of elements of $\pk$.

For a subset $E\subset\mathbb R^d$, we denote the lower box dimension of $E$ and the upper box dimension of $E$
by $\dboxinf(E)$ and $\dboxsup(E)$, respectively. Also, for a probability measure $\mu$, we define the \emph{small} and \emph{big lower (resp. upper) multifractal box dimensions} of $\mu$ by

$$\begin{array}{rclcrcl}
\displaystyle \dinfsmallmu&=&\inf_{\mu(E)>0}\dboxinf(E)&&\displaystyle \dinfbigmu&=&\lim_{\veps>0}\inf_{\mu(E)>1-\veps}\dboxinf(E)\\
\displaystyle \dsupsmallmu&=&\inf_{\mu(E)>0}\dboxsup(E)&&\displaystyle \dsupbigmu&=&\lim_{\veps>0}\inf_{\mu(E)>1-\veps}\dboxsup(E).
\end{array}$$

Finally, we define the \emph{local upper box dimension of $K$} by
$$\dboxsuploc(K)=\inf_{x\in K}\inf_{r>0}\dboxsup\big(K\cap B(x,r)\big).$$

\begin{MYJAK}
Let $K$ be a compact subset of $\mathbb R^d$. Then a typical measure $\mu\in\mathcal P(K)$ satisfies
$$\dinfsmallmu=\dinfbigmu=0$$
$$\dboxsuploc(K)\leq \dsupsmallmu\leq \dsupbigmu\leq \dboxsup(K).$$
\end{MYJAK}

The result concerning the upper multifractal box dimension does not solve completely the problem for compact sets as simple as $K=\{0\}\cup[1,2]$. In this
case, we just obtain that, typically
$$0\leq \dsupsmallmu\leq\dsupbigmu\leq 1.$$
In particular, we do not know whether the interval $[0,1]$ is the shortest possible, or whether $\dsupsmallmu$ and $\dsupbigmu$ coincide for a typical measure.

\smallskip

Our aim, when we began this work, was to solve this question. To answer it, we need to introduce the maximal local upper box dimension of a set $E$. It is defined
by 
$$\dsuplocmax(E)=\sup_{y\in E,\rho>0}\dboxsuploc\big(E\cap B(y,\rho)\big).$$
Our first main result now reads:
\begin{theorem}\label{THMMRMIEUX}
Let $K$ be a compact subset of $\mathbb R^d$. Then a typical measure $\mu\in\mathcal P(K)$ satisfies
\begin{eqnarray*}
\dsupsmallmu&=&\dboxsuploc(K)\\
\dsupbigmu&=&\dsuplocmax(K).\\
\end{eqnarray*}
\end{theorem}
If we apply this theorem with $K=\{0\}\cup[1,2]$, then we find that a typical measure $\mu\in\pk$ satisfies
$$\dsupsmallmu=0\textrm{ and }\dsupbigmu=1.$$

\subsection{Multifractal box dimensions}
In \cite{OL11}, L. Olsen has put the work of Myjak and Rudnicki in a more general context, that of multifractal box dimensions of measures,
which is interesting by itself. Fix a Borel probability measure $\pi$ on $\mathbb R^d$ with support $K$. For a bounded subset $E$ of $K$,
the \emph{multifractal box dimensions} of $E$ with respect to $\pi$ are defined as follows. For $r>0$ and a real number $q$, write
$$\npiq(E,r)=\inf_{(B(x_i,r))\textrm{ is a cover of }E}\sum_i \pi\big(B(x_i,r)\big)^q.$$
The \emph{lower} and \emph{upper covering multifractal box dimensions} of $E$ of order $q$ with respect to $\pi$ are defined by
\begin{eqnarray*}
\dinfpiq(E)&=&\liminf_{r\to 0}\frac{\log\npiq(E,r)}{-\log r}\\
\dsuppiq(E)&=&\limsup_{r\to 0}\frac{\log\npiq(E,r)}{-\log r}.
\end{eqnarray*}

Let now $\mu\in\pk$. We define the \emph{small} and \emph{big lower multifractal box dimensions} of $\mu$ of order $q$ with respect to the measure
$\pi$ (resp. the \emph{small} and \emph{big upper multifractal box dimensions} of $\mu$ of order $q$ with respect to the measure
$\pi$) by 
$$\begin{array}{rclcrcl}
\dinfsmallpiq(\mu)&=&\inf_{\mu(E)>0}\dinfpiq(E)&\quad&\dinfbigpiq(\mu)&=&\lim_{\veps>0}\inf_{\mu(E)>1-\veps}\dinfpiq(E)\\[0.2cm]
\dsupsmallpiq(\mu)&=&\inf_{\mu(E)>0}\dsuppiq(E)&\quad&\dsupbigpiq(\mu)&=&\lim_{\veps>0}\inf_{\mu(E)>1-\veps}\dsuppiq(E).
\end{array}$$

Multifractal box dimensions of measures play a central role in multifractal analysis. For instance, the multifractal box dimensions of measures in $\mathbb R^d$
having some degree of self-similarity have been intensively studied (see \cite{Fal97} and the references therein). In \cite{OL11}, L. Olsen give estimations
of the typical multifractal box dimensions of measures, in the spirit of Myjak and Rudnicki. To state his result, we need a few definitions. Firstly, the 
\emph{upper moment scaling} of $\pi$ is the function $\tau_\pi:\mathbb R\to\mathbb R$ defined by
\begin{eqnarray*}
\tau_\pi(q)&=&\dsuppiq(K).
\end{eqnarray*}
The \emph{local upper multifractal box dimension of $K$ of order $q$} is defined by 
$$\dsuppilocq(K)=\inf_{x\in K}\inf_{r>0}\dsuppiq\big(K\cap B(x,r)\big).$$
This last quantity will be also called the \emph{local upper moment scaling} of $\pi$ and will be denoted by $\taupiloc(q)$.
Finally, let 
\begin{eqnarray*}
\Dsuppi(-\infty)&=&\limsup_{r\to 0}\frac{\log \inf_{x\in K}\pi\big(B(x,r)\big)}{\log r}\\
\Dinfpi(+\infty)&=&\liminf_{r\to 0}\frac{\log \sup_{x\in K}\pi\big(B(x,r)\big)}{\log r}.
\end{eqnarray*}
Recall also that a measure $\pi$ on $\mathbb R^d$ is called a \emph{doubling measure} provided there exists $C>0$ such that 
$$\sup_{x\in\supp(\pi)}\sup_{r>0}\frac{\pi\big(B(x,2r)\big)}{\pi\big(B(x,r)\big)}\leq C.$$
We can now give Olsen's result.
\begin{OLSEN}
Let $\pi$ be a Borel probability measure on $\mathbb R^d$ with compact support $K$. 
\begin{enumerate}
\item A typical measure $\mu\in\pk$ satisfies 
$$
\begin{array}{rcccccll}
-q\Dinfpi(+\infty)&\leq&\dinfsmallpiq(\mu)&\leq&\dinfbigpiq(\mu)&\leq&-q\Dsuppi(-\infty)&\quad\textrm{for all }q\leq 0,\\
-q\Dsuppi(-\infty)&\leq&\dinfsmallpiq(\mu)&\leq&\dinfbigpiq(\mu)&\leq&-q\Dinfpi(+\infty)&\quad\textrm{for all }q\geq 0.
\end{array}
$$
\item If $\pi$ is a doubling measure, then a typical measure $\mu\in\pk$ satisfies
$$\taupiloc(q)\leq\dsupsmallpiq(\mu)\leq \dsupbigpiq(\mu)\leq\taupi(q)\textrm{ for all }q\leq 0.$$
If moreover $K$ does not contain isolated points, then this result remains true for all $q\in\mathbb R$.
\end{enumerate}
\end{OLSEN}
Putting $q=0$, this implies in particular Myjak and Rudnicki's theorem.
†
\subsection{Statement of our main results}
Of course, the questions asked after Theorem A have also a sense in this more general context. To answer them, we have to introduce
the \emph{maximal local upper moment scaling} of $\pi$ which is defined by
$$\taupilocmax(q)=\sup_{y\in K,\rho>0}\dsuppilocq\big(K\cap B(y,\rho)\big).$$
\begin{theorem}\label{THMMAINBIG}
Let $\pi$ be a doubling Borel probability measure on $\mathbb R^d$ with compact support $K$. Then a typical measure $\mu\in\pk$ satisfies, for any $q\in\mathbb R$,
\begin{eqnarray*}
\dsupsmallpiq(\mu)&=&\taupiloc(q)\\
\dsupbigpiq(\mu)&=&\taupilocmax(q).
\end{eqnarray*}
\end{theorem}
Putting $q=0$, we retrieve Theorem \ref{THMMRMIEUX}.

\medskip

We can also observe that Olsen's theorem does not settle completely the typical values of the lower multifractal box dimensions.
For instance, when computed for a self-similar compact set $K$ satisfying the open set condition (see below) 
and an associated self-similar measure $\pi$, the values of $\Dinfpi(+\infty)$ and $\Dsuppi(-\infty)$ are in general different.
Moreover, it has been pointed out in \cite{BAYLQ} that, given a fixed compact set $K\subset\mathbb R^d$, a typical probability
measure $\pi\in\pk$ satisfies $\Dsuppi(-\infty)=+\infty$ and $\Dinfpi(+\infty)=0$!

We have been able to compute the typical value of the big lower multifractal box dimension of a measure.
 As before, we need to introduce new definitions, which are uniform versions of $\Dsuppi(-\infty)$ and $\Dinfpi(+\infty)$. Let $\pi$ be a Borel probability measure
with support $K$. Define
$$\Dsuppiunif(-\infty)=\inf_N\inf_{\substack{y_1,\dots,y_N\in K\\\rho>0}}\limsup_{r\to 0}\inf_{i=1,\dots,N}\frac{\log\big(\inf_{x\in B(y_i,\rho)}\pi(B(x,r))\big)}{\log r}$$
$$\Dinfpiunif(+\infty)=\sup_N\sup_{\substack{y_1,\dots,y_N\in K\\\rho>0}}\liminf_{r\to 0}\sup_{i=1,\dots,N}\frac{\log\big(\sup_{x\in B(y_i,\rho)}\pi(B(x,r))\big)}{\log r}$$

\begin{theorem}\label{THMMAINSMALL}
Let $\pi$ be a Borel probability measure with compact support $K$. Then a typical measure $\mu\in\pk$ satisfies
$$\dinfbigpiq(\mu)=
\begin{cases}
-q\Dsuppiunif(-\infty)&\textrm{provided } q\geq 0\\
-q\Dinfpiunif(+\infty)&\textrm{provided } q\leq 0.
\end{cases}
$$
\end{theorem}
Unfortunately, we did not find a similar result for the small lower multifractal box dimensions. We have just been able
to improve Olsen's inequality. This improvement is sufficient to conclude for self-similar compact sets. We need to introduce
the following quantities. Let $\pi$ be a Borel probability measure with compact support $K$. Define
\begin{eqnarray*}
 \Dsuppiunifmax(-\infty)&=&\sup_{\substack{z\in K\\ \kappa>0}}\inf_{\substack{y_1,\dots,y_N\in B(z,\kappa)\\\rho>0}}\limsup_{r\to 0}\inf_{i=1,\dots,N}\frac{\log\big(\inf_{x\in B(y_i,\rho)}\pi(B(x,r))\big)}{\log r}\\
\Dsuppimax(-\infty)&=&\sup_{\substack{y\in K\\ \rho>0}}\limsup_{r\to 0}\frac{\log \inf_{x\in B(y,\rho)}\pi\big(B(x,r)\big)}{\log r}\\
\Dinfpiunifmin(+\infty)&=&\inf_{\substack{z\in K\\ \kappa>0}}\sup_{\substack{y_1,\dots,y_N\in B(z,\kappa)\\\rho>0}}\liminf_{r\to 0}\sup_{i=1,\dots,N}\frac{\log\big(\sup_{x\in B(y_i,\rho)}\pi(B(x,r))\big)}{\log r}\\
\Dinfpimin(+\infty)&=&\inf_{\substack{y\in K\\ \rho>0}}\liminf_{r\to 0}\frac{\log \sup_{x\in B(y,\rho)}\pi\big(B(x,r)\big)}{\log r}\\
\end{eqnarray*}
\begin{theorem}\label{THMMAINSMALLBIS}
Let $\pi$ be a Borel probability measure with compact support $K$. Then a typical measure $\mu\in\pk$ satisfies
$$-q\Dsuppimax(-\infty)\leq \dinfsmallpiq(\mu)\leq -q\Dsuppiunifmax(-\infty)\textrm{ provided } q\geq 0$$
$$-q\Dinfpimin(+\infty)\leq \dinfsmallpiq(\mu)\leq -q\Dinfpiunifmin(-\infty)\textrm{ provided } q\leq 0.$$
\end{theorem}
\smallskip

Although they are not very engaging, the above quantities can be easily computed for regular measures $\pi$. This is for instance the case for
self-similar measures on self-similar compact sets. Fix an integer $M\geq 2$. For any $m=1,\dots,M$, let $S_m:\mathbb R^d\to\mathbb R^d$
be a contracting similarity with Lipschitz constant $r_m\in(0,1)$. Let $(p_1,\dots,p_M)$ be a probability vector. We define $K$ and $\pi$ as the self-similar 
compact set and the self-similar measure associated with the list $(S_1,\dots,S_M,p_1,\dots,p_M)$, i.e. $K$ is the unique nonempty compact subset of
$\mathbb R^d$ such that
$$K=\bigcup_m S_m(K),$$
and $\pi$ is the unique Borel probability measure on $\mathbb R^d$ such that
$$\pi=\sum_m p_m\pi\circ S_m^{-1}$$
(see for instance \cite{Fal97}). It is well known that $\supp\pi=K$. We say that the list $(S_1,\dots,S_M)$ satisfies the \emph{Open Set Condition} if there exists an
open and nonempty bounded subset $U$ of $\mathbb R^d$ with $S_mU\subset U$ for all $m$ and $S_m U\cap S_l U=\varnothing$ for all $l,m$ with $l\neq m$.

Theorem \ref{THMMAINSMALL} and Theorem \ref{THMMAINSMALLBIS} imply the following more appealing corollary:
\begin{corollary}\label{CORSELFSIMILAR}
Let $K$ and $\pi$ as above, and assume that the Open Set Condition is satisfied. Let 
$$s_{\min}=\min_m \frac{\log p_m}{\log r_m}\textrm{ and }s_{\max}=\max_m \frac{\log p_m}{\log r_m}.$$
Then a typical measure $\mu\in\pk$ satisfies
$$\dinfsmallpiq(\mu)=\dinfbigpiq(\mu)=\begin{cases}
-s_{\max}q&\textrm{ for any }q\geq 0\\
-s_{\min}q&\textrm{ for any }q\leq 0.
\end{cases}
$$
\end{corollary}
This improves Theorem 2.1 of \cite{OL11} which just says that a typical $\mu\in\pk$ satisfies
$$-s_{\max}q\leq \dinfsmallpiq(\mu)\leq \dinfbigpiq(\mu)\leq -s_{\min }q\textrm{ for all }q\geq 0$$
$$-s_{\min}q\leq \dinfsmallpiq(\mu)\leq \dinfbigpiq(\mu)\leq -s_{\max }q\textrm{ for all }q\leq 0.$$
\subsection{Organization of the paper}
The paper is organized as follows. In Section 2, we summarize all the results which will be needed throughout the paper.
Section 3 is devoted to the proof of Theorem \ref{THMMAINBIG}. The proofs of Theorems  \ref{THMMAINSMALL}
and \ref{THMMAINSMALLBIS} share some similarities. They will be exposed in Section 4, as well as the application
to self-similar measures.

\section{Preliminaries} \label{SECPRELIMINARIES}

Throughout this paper, $\mathcal P(K)$ will be endowed with the weak topology. It is well known (see for instance \cite{Par67})
that this topology is completely metrizable by the Fortet-Mourier distance defined as follows. Let $\textrm{Lip}(K)$ denote the family 
of Lipschitz functions $f:K\to\mathbb R$, with $|f|\leq 1$ and $\textrm{Lip}(f)\leq 1$, where $\textrm{Lip}(f)$ denotes
the Lipschitz constant of $f$. The metric $L$ is defined by
$$L(\mu,\nu)=\sup_{f\in\textrm{Lip}(K)}\left|\int fd\mu-\int fd\nu\right|$$
for any $\mu,\nu\in\mathcal P(K)$. We endow $\mathcal P(K)$ with the metric $L$. In particular, for $\mu\in\mathcal P(K)$
and $\delta>0$, $B_{L}(\mu,\delta)=\{\nu\in\mathcal P(K);\ L(\mu,\nu)<\delta\}$ will stand for the ball with center at $\mu$ and radius
equal to $\delta$.

We shall use several times the following lemma.
\begin{lemma}\label{LEMTOPO1}
 For any $\alpha\in(0,1)$, for any $\beta>0$, there exists $\eta>0$ such that, for any $E$ a Borel subset of $K$, for any
$\mu,\nu\in\mathcal P(K)$, 
$$L(\mu,\nu)<\eta\implies \mu(E)\leq\nu\big(E(\alpha)\big)+\beta,$$
where $E(\alpha)=\{x\in K;\ \textrm{dist}(x,E)<\alpha\}$.
\end{lemma}
\begin{proof}
 We set
$$f(t)=\left\{
\begin{array}{ll}
 \alpha&\textrm{provided }t\in \overline E\\
\alpha-\textrm{dist}(x,E)&\textrm{provided }0<\textrm{dist}(x,E)\leq\alpha\\
0&\textrm{otherwise}.
\end{array}\right.$$
Then $f$ is Lipschitz, with $|f|\leq 1$ and $\textrm{Lip}(f)\leq 1$. Thus, provided $L(\mu,\nu)<\eta$,
\begin{eqnarray*}
 \mu(E)&\leq&\frac 1\alpha\int fd\mu\\
&\leq&\frac 1\alpha \left[\int fd\nu+\eta\right]\\
&\leq&\nu \big(E(\alpha)\big)+\frac\eta\alpha.
\end{eqnarray*}
Hence, it suffices to take $\eta=\alpha\beta$.
\end{proof}

An application of Lemma \ref{LEMTOPO1} is the following result on open subsets of $\pk$:
\begin{lemma}\label{LEMTOPOOPEN}
 Let $x\in K$, $a\in\mathbb R$ and $r>0$. Then $\{\mu\in\pk;\ \mu\big(B(x,r)\big)>a\}$ is open.
\end{lemma}
\begin{proof}
  If $a$ does not belong to $[0,1)$, then the set is either empty or equal to $\pk$. Otherwise, let $\mu\in\pk$ be such that $\mu\big(B(x,r)\big)>a$. One may find $\veps>0$ such that 
$\mu\big(B(x,(1-\veps)r)\big)>a$. Thus the result follows from Lemma \ref{LEMTOPO1} applied with $E=B(x,(1-\veps)r)$,
$\alpha=\veps r$ and $\beta=\big(\mu\big(B(x,(1-\veps)r)\big)-a\big)/2$.
\end{proof}

Finally, we will need that some subsets of $\pk$ are dense in $\pk$. The result that we need can be found e.g. in \cite[Lemma 2.2.4.]{OL05}.
\begin{lemma}\label{LEMDENS1}
 Let $(x_n)_{n\geq 1}$ be a dense sequence of $K$ and let $(\eta_n)_{n\geq 1}$ be a sequence of positive real numbers going to zero. For each $n\geq 1$ and each $i\in\{1,\dots,n\}$, let $\mu_{n,i}\in \mathcal P(K)$ be such that
$\supp(\mu_{n,i})\subset K\cap B(x_i,\eta_n)$. Then, for any $m\geq 1$, $\bigcup_{n\geq m}\big\{\sum_{i=1}^n p_i\mu_{n,i};\ p_i> 0,\ \sum_i p_i=1\big\}$
is dense in $\pk$.
\end{lemma}

\section{The typical upper multifractal box dimensions}
This section is devoted to the proof of Theorem \ref{THMMAINBIG}
\subsection{Packing, covering and doubling measures}
When $\pi$ is a doubling measure, it will be convenient to express the multifractal box dimensions of a set using packings instead of coverings. For $E\subset\mathbb R^d$, recall that a family of balls $\big(B(x_i,r)\big)$ is called a \emph{centred packing} of $E$ if $x_i\in E$ for all $i$ and $|x_i-x_j|>2$ for all $i\neq j$.
We then define
$$\ppiq(E,r)=\sup_{(B(x_i,r))\textrm{ is a packing of }E}\sum_i \pi\big(B(x_i,r)\big)^q.$$
When $\pi$ is a doubling measure, $\dinfpiq(E)$ and $\dsuppiq(E)$ can be defined using packings (see \cite{OL11}):
\begin{lemma}
Let $\pi$ be a doubling Borel probability measure on $\mathbb R^d$ with support $K$. Then
\begin{eqnarray*}
\dinfpiq(E)&=&\liminf_{r\to 0}\frac{\log \ppiq(E)}{-\log r}\\
\dsuppiq(E)&=&\limsup_{r\to 0}\frac{\log \ppiq(E)}{-\log r}
\end{eqnarray*}
for all $E\subset K$ and all $q\in\mathbb R$.
\end{lemma}
One of the advantages of using packings instead of covering is that it helps us to obtain regularity of the map $q\mapsto \dsuppiq(E)$, as shown in the following lemma.
\begin{lemma}
Let $\pi$ be a doubling Borel probability measure on $\mathbb R^d$ with support $K$, and let $E\subset K$. 
\begin{enumerate}
\item The map $q\mapsto \dsuppiq(E)$ is nonincreasing, convex and therefore continuous.
\item The maps $q\mapsto \taupiloc(q)$ and $q\mapsto\taupilocmax(q)$ are nonincreasing.
\end{enumerate}
\end{lemma}
\begin{proof}
Part (1) is Lemma 4.2 of \cite{OL11} and Part (2) is trivial.
\end{proof}
As a first application, we show that, in order to find a residual subset $\mathcal R$ of $\pk$ such that any $\mu\in\mathcal R$ satisfies
the conclusions of Theorem \ref{THMMAINBIG} for any $q\in\mathbb R$, it suffices to find a residual subset which works for a fixed $q\in\mathbb R$.
\begin{proposition}
Let $\pi$ be a doubling Borel probability measure on $\mathbb R^d$ with support $K$. Then there exists a countable set $\mathbf Q\subset\mathbb R$ such that 
\begin{eqnarray*}
\bigcap_{q\in\mathbb R}\big\{\mu\in\pk;\ \taupiloc(q)\leq\dsupsmallpiq(\mu)\big\}&=&\bigcap_{q\in\mathbf Q}\big\{\mu\in\pk;\ \taupiloc(q)\leq\dsupsmallpiq(\mu)\big\}\\
\bigcap_{q\in\mathbb R}\big\{\mu\in\pk;\ \taupiloc(q)\geq\dsupsmallpiq(\mu)\big\}&=&\bigcap_{q\in\mathbf Q}\big\{\mu\in\pk;\ \taupiloc(q)\geq\dsupsmallpiq(\mu)\big\}\\
\bigcap_{q\in\mathbb R}\big\{\mu\in\pk;\ \taupilocmax(q)\leq\dsupbigpiq(\mu)\big\}&=&\bigcap_{q\in\mathbf Q}\big\{\mu\in\pk;\ \taupilocmax(q)\leq\dsupbigpiq(\mu)\big\}\\
\bigcap_{q\in\mathbb R}\big\{\mu\in\pk;\ \taupilocmax(q)\geq\dsupbigpiq(\mu)\big\}&=&\bigcap_{q\in\mathbf Q}\big\{\mu\in\pk;\ \taupilocmax(q)\geq\dsupbigpiq(\mu)\big\}\end{eqnarray*}
\end{proposition}
\begin{proof}
Let $\mathbf Q_1$ (resp. $\mathbf Q_2$) be the set of points of discontinuity of $\taupiloc$ (resp. of $\taupilocmax$). $\mathbf Q_1$ and $\mathbf Q_2$ are at most countable.
Set $\mathbf Q=\mathbf Q_1\cup\mathbf Q_2\cup\mathbb Q$.

\smallskip

The first equality is already contained in \cite[Prop. 4.3]{OL11}. Regarding the second one, let $\mu\in\pk$ be such that $\taupiloc(q)\geq\dsupsmallpiq(\mu)$ for any $q\in\mathbf Q$, 
and let us fix $q\in\mathbb R\backslash\mathbf Q$. Let $(q_n)$ be a sequence of $\mathbb Q$ increasing to $q$. For each $n$, we may find $E_n$ with
$\mu(E_n)>0$ and $\overline{\dim}_{\pi,\rm B}^{q_n}(E_n)\leq \taupiloc(q_n)+\frac1n$. For $n$ large enough, we get, by continuity of $\taupiloc$ at $q$,
$$\dsuppiq(E_n)\leq \overline{\dim}_{\pi,\rm B}^{q_n}(E_n)\leq \taupiloc(q)+\delta$$ for any fixed $\delta>0$, so that
$\dsupsmallpiq(\mu)\leq \taupiloc(q).$

The proof of the third inequality goes along the same lines and is left to the reader. Regarding the last one, let $\mu\in\pk$ be such that 
$\taupilocmax(q)\geq\dsupbigpiq(\mu)$ for any $q\in\mathbf Q$ and let us fix $q\in\mathbb R$. Let $\veps>0$, $\delta>0$ and let $(q_n)\subset\mathbf Q$ 
be a sequence decreasing to $q$. Let also $(\veps_n)\subset(0,+\infty)$ be such that $\sum_n\veps_n<\veps$. For each $n$, we may find $E_n\subset K$
such that $\mu(E_n)>1-\veps_n$ and $\overline{\dim}_{\pi,\rm B}^{q_n}(E_n)\geq\taupilocmax(q_n)+\delta$. Set $E=\bigcap_n E_n$ so that
$\mu(E)>1-\veps$ and observe that, by continuity of $\dsuppiq(E)$, 
$$\dsuppiq(E)\leq \liminf_n\taupilocmax(q_n)+\delta\leq\taupilocmax(q)+\delta.$$
\end{proof}
We conclude this section by pointing out that, working with doubling measures, we can also add a dilation factor when studying the multifractal dimensions.
\begin{lemma}
Let $\pi$ be a doubling Borel probability measure on $\mathbb R^d$ with compact support $K$. Let $c>0$, $E\subset K$ and $q\in\mathbb R$.
Then 
$$\dsuppiq(E)=\limsup_{r\to 0}\frac{\log \sup_{(B(x_i,r))\textrm{ \rm is a packing of }E}\sum_i \pi\big(B(x_i,cr)\big)^q}{-\log r}.$$
\end{lemma}

\subsection{The lower bounds}
In this subsection, we fix $q\in\mathbb R$. We shall prove, in the same time, that quasi-all measures $\mu\in\pk$ satisfy
\begin{description}
\item[(A)] $\dsupsmallpiq(\mu)\geq \taupiloc(q)$;
\item[(B)] $\dsupbigpiq(\mu)\geq \taupilocmax(q)$
\end{description}
(we shall prove $\bf(A)$ since we want to dispense with the assumption ``$K$ has no isolated points"). 

If we want to prove
$\bf (A)$, we consider $t<\taupiloc(q)$ and we set $F=G=K$. If we want to prove $\bf (B)$, then we consider $t<\taupilocmax(q)$,
a pair $(y,\kappa)\in K\cap (0,+\infty)$ such that $\dsuppilocq\big(B(y,\kappa)\cap K\big)>t$, and we set 
$F=K\cap B(y,\kappa)$, $G=K\cap B(y,\kappa/2)$.

Let now $x\in K$ and $s>0$.
\begin{itemize}
\item if $x\notin F$, then we set $\mu_{x,s}=\delta_x$ and $r_{x,s}=s$;
\item if $x\in F$, then $\dsuppiq\big(B(x,s)\cap F\big)>t$, so that we may choose $r_{x,s}\in(0,s)$ satisfying
$$t<\frac{\log \ppiq\big(B(x,s)\cap F,r_{x,s}\big)}{-\log r_{x,s}}.$$
Thus, there exists a finite set $\Lambda_{x,s}\subset B(x,s)\cap F$ which consists in points at distance at least $2r_{x,s}$ and satisfying
$$\sum_{z\in\Lambda_{x,s}}\pi\big(B(z,r_{x,s})\big)^q\geq r_{x,s}^{-t}.$$
We then set
$$\mu_{x,s}=\frac1{\sum_{z\in\Lambda_{x,s}}\pi\big(B(z,r_{x,s})\big)^q}\sum_{z\in\Lambda_{x,s}}\pi\big(B(z,r_{x,s})\big)^q\delta_z.$$
\end{itemize}
Observe that, in both cases, $\supp(\mu_{x,s})\subset B(x,s)$.

\medskip

Let us introduce some notations. We denote by $\mathcal F$ the set of nonempty finite subsets of $K$. For $A\in \mathcal F$, we denote by 
$$\mathcal Q(A)=\left\{(p_x)_{x\in A};\ p_x\in(0,1),\ \sum_{x\in A}p_x=1\right\}.$$
Next, for $A\in\mathcal F$ and $\mathbf p=(p_x)_{x\in A}\in\mathcal Q(A)$, we denote 
\begin{eqnarray*}
\mu_{A,\mathbf p,s}&=&\sum_{x\in A}p_x\mu_{x,s}\\
r_{A,s}&=&\inf_{x\in A}r_{x,s}\in(0,s).
\end{eqnarray*}

An application of Lemma \ref{LEMDENS1} shows that, for any sequence $(\eta_n)$ decreasing to zero and for any $m\geq 1$,
$$\bigcup_{n\geq m}\bigcup_{A\in\mathcal F}\bigcup_{\mathbf p\in\mathcal Q(A)}\big\{\mu_{A,\mathbf p,\eta_n}\big\}$$
is dense in $\pk$. Finally, for any $A\in\mathcal F$, any $\mathbf p\in\mathcal Q(A)$, any $s>0$ and any $\veps>0$, we consider a real number
$\eta_{A,\mathbf p,s,\veps}>0$ such that any $\mu\in\pk$ satisfying $L(\mu,\mu_{A,\mathbf p,s})<\eta_{A,\mathbf p,s,\veps}$ also verifies, 
for any $E\subset K$, 
$$\mu_{A,\mathbf p,s}\big(E(r_{A,s}/2)\big)\geq\mu(E)-\veps.$$
We now set
\begin{eqnarray*}
\mathcal R&=&\bigcap_{m\geq 1}\bigcup_{n\geq m}\bigcup_{A\in\mathcal F}\bigcup_{\mathbf p\in\mathcal Q(A)}B_L\big(\mu_{A,\mathbf p,1/n},\eta_{A,\mathbf p,1/n,1/n}\big)\\
&&\quad\bigcap \big\{\mu\in\pk;\ \mu(G)>0\big\}.
\end{eqnarray*}
$\mathcal R$ is a dense $G_\delta$-subset of $\pk$ and we pick $\mu\in\mathcal R$. We shall prove that either
\begin{description}
\item[(A)] $\dsupsmallpiq(\mu)\geq t$
\end{description}
or
\begin{description}
\item[(B)] $\dsupbigpiq(\mu)\geq t$.
\end{description}
In case $\bf(A)$, let $E\subset K$ with $\mu(E)>0$ and let $E'= E$. In case $\bf (B)$, we begin by fixing $\veps>0$ such that any subset $E$ of $K$ satisfying
$\mu(E)\geq1-\veps$ also satisfies $\mu(E\cap G)>0$. Then, let $E\subset K$ with $\mu(E)\geq 1-\veps$ and let us define $E'=E\cap G$. In both cases, we
are going to show that
$\dsuppiq(E')\geq t.$

\medskip

Since $\mu\in\mathcal R$ we may find sequences $(A_n)\subset\mathcal F$, $(\mathbf p_n)$ with $\mathbf p_n\in\mathcal Q(A_n)$, and $(s_n)$ going to zero
such that 
$$\mu\in B_L\big(\mu_{A_n,\mathbf p_n,s_n},\eta_{A_n,\mathbf p_n,s_n,s_n}\big).$$
For commodity reasons, we set $r_n=r_{A_n,s_n}$, $\eta_n=\eta_{A_n,\mathbf p_n,s_n,s_n}$ and $E'_n=E'(r_n/2)$. Our assumption on $\eta_n$ ensures that
$$\mu_{A_n,\mathbf p_n,s_n}(E'_n)\geq \mu(E')-s_n\geq \frac12\mu(E')$$
provided $n$ is large enough. By construction of $\mu_{A_n,\mathbf p_n,s_n}$, we may find $x_n\in A_n$ such that $\mu_{x_n,s_n}(E'_n)\geq\frac 12\mu(E').$
Moreover, $x_n$ also belongs to $F$. This is clear in case {\bf (A)} and in case {\bf (B)}, this follows from
$$\mu_{x_n,s_n}(E'_n)\leq \delta_{x_n}\big(G(\kappa/2)\big)=\delta_{x_n}(F)=0,$$
provided $x_n\notin F$ and $n$ is large enough so that $E'_n\subset G(\kappa/2)$. Hence, by definition of $\mu_{x_n,s_n}$ when $x_n\in F$, we obtain
\begin{eqnarray*}
\sum_{z\in\Lambda_{x_n,s_n}\cap E'_n}\pi\big(B(z,r_{x_n,s_n})\big)^q&\geq&\frac 12\mu(E')\left(\sum_{z\in \Lambda_{x_n,s_n}}\pi\big(B(z,r_{x_n,s_n})\big)^q\right)\\
&\geq&\frac 12\mu(E')r_{x_n,s_n}^{-t}.
\end{eqnarray*}

Now for any $z\in\Lambda_{x_n,s_n}\cap E'_n$, there exists $x_z\in E$ with $\|x_z-z\|\leq\frac12r_n\leq r_{x_n,s_n}$.
It is then not hard to show that $\big(B(x_z,r_{x_n,s_n}/2)\big)_{z\in\Lambda_{x_n,s_n}\cap E'_n}$ is a centred packing of $E$.
Indeed, for $u\neq v$ in $\Lambda_{x_n,s_n}$, 
\begin{eqnarray*}
\|x_u-x_v\|&\geq&\|u-v\|-\|u-x_u\|-\|v-x_v\|\\
&\geq&2r_{x_n,s_n}-\frac{r_{x_n,s_n}}2-\frac{r_{x_n,s_n}}2={r_{x_n,s_n}}.
\end{eqnarray*}

We also observe that, for any $z\in\Lambda_{x_n,s_n}\cap E'_n$, 
$$B(x_z,r_{x_,s_n}/2)\subset B(z,r_{x_n,s_n})\subset B(x_z,2r_{x_n,s_n}).$$
Summarizing what we have done, this means that we have found a packing $\big(B(u,r)\big)_{u\in\Lambda}$ of $E'$ with $r$ as small as we want, and a constant
$c_0\in\mathbb R$ ($c_0=2$ if $q\geq 0$, $c_0=1/2$ if $q\leq 0$) so that
$$\sum_{u\in\Lambda}\pi\big(B(u,c_0r)\big)^q\geq\frac12\mu(E')r^{-t}.$$
This yields $\dsuppiq(E')\geq t$ and this concludes this part of the proof.

\subsection{The upper bounds}

We now turn to the proof of the upper bounds in Theorem \ref{THMMAINBIG} which are simpler. As before, we fix $q\in\mathbb R$. We first show that a generic
$\mu\in\pk$ satisfies 
$$\dsupsmallpiq(\mu)\leq \taupiloc(q).$$
Indeed, let $t>\taupiloc(q)$. There exists $x_t\in K$ and $r_t>0$ such that
$$\dsuppiq\big(B(x_t,r_t)\big)\leq t.$$
We set $\mathcal U_t=\big\{\mu\in\pk;\ \mu\big(B(x_t,r_t)\big)>0\big\}$. $\mathcal U_t$ is dense and open. Moreover, any $\mu\in\mathcal U_t$ satisfies
$\dsupsmallpiq(\mu)\leq t$. The residual set we are looking for is thus given by
$$\mathcal R=\bigcap_{t\in\mathbb Q,t>\taupiloc(q)}\mathcal U_t.$$
We now show that a generic $\mu\in\pk$ satisfies
$$\dsupbigpiq(\mu)\leq \taupilocmax(q).$$
As before, let $t>\taupilocmax(q)$. We just need to prove that a generic $\mu\in\pk$ satisfies $\dsupbigpiq(\mu)\leq t$.
 Let $(y_n)$ be a dense sequence of distinct points in $K$, and let $(\kappa_n)$ be a sequence decreasing to zero. For each $n$, we may find $x_n\in B(y_n,\kappa_n)$ and 
$r_n>0$ such that $\dsuppiq\big(B(x_n,r_n)\big)\leq t$. We may assume that the sequence $(r_n)$ is going to zero.

 We set 
$$\Lambda_n=\left\{\sum_{i=1}^n p_i\delta_{x_i};\ \sum_{i=1}^n p_i=1, p_i>0\right\},$$
so that, by Lemma \ref{LEMDENS1}, $\bigcup_{n\geq m}\Lambda_n$ is dense for any integer $m\geq 1$. Moreover, Lemma \ref{LEMTOPO1} tells us that, 
for any $m\geq 1$, one may find $\eta_m>0$ such that, for any $\mu\in \Lambda_n$, for any $\nu\in\pk$ with $L(\mu,\nu)<\eta_m,$
$$\nu\left(\bigcup_{i=1}^n B(x_i,r_n)\right)\geq \mu\left(\bigcup_{i=1}^m B(x_i,r_n/2)\right)-\frac 1m\geq 1-\frac 1m.$$
We then set
$$\mathcal R=\bigcap_{m\geq 1}\bigcup_{n\geq m}\bigcup_{\mu\in \Lambda_n}B_L(\mu,\eta_m).$$
$\mathcal R$ is a dense $G_\delta$-set. Pick $\nu\in\mathcal R$ and $\veps>0$. Let also $m\geq 1$ with $\frac 1m\leq\veps$. We may find $n\geq m$ and $\mu\in E_n$ such that $L(\mu,\nu)<\eta_m$. Thus, setting $E=\bigcup_{i=1}^n B(x_i,r_n)$, we get
$$\left\{
\begin{array}{rcl}
\nu(E)&\geq &1-\frac 1m\geq 1-\veps\\
\dsuppiq(E)&\leq& t.
\end{array}\right.$$
Therefore, $\dsupbigpiq(\nu)\leq t$.

\section{The typical lower multifractal box dimensions}
This section is devoted to the proof of Theorem \ref{THMMAINSMALL}. We begin with a lemma which helps us to avoid the assumption
"$\pi$ is a doubling measure" throughout the proofs.
\begin{lemma}\label{LEMNODOUBLING}
Let $\pi$ be a Borel probability measure with compact support $K$. Then 
$$\Dsuppiunif(-\infty)=\inf_N\inf_{\substack{y_1,\dots,y_N\in K\\\rho>0}}\limsup_{r\to 0}\inf_{i=1,\dots,N}\frac{\log\big(\inf_{B(x,r)\cap B(y_i,\rho)\neq\varnothing}\pi(B(x,r))\big)}{\log r}$$
$$\Dsuppimax(-\infty)=\sup_{\substack{y\in K\\ \rho>0}}\limsup_{r\to 0}\frac{\log \inf_{B(x,r)\cap B(y,\rho)\neq\varnothing}\pi\big(B(x,r)\big)}{\log r}.$$
\end{lemma}
\begin{proof}
Let $t>\Dsuppiunif(-\infty)$. One may find $y_1,\dots,y_N\in K$, $\rho>0$, $\alpha>0$ such that, for any $r\in(0,\alpha)$,
there exists $i\in\{1,\dots,N\}$ such that any $x\in B(y_i,\rho)$ satisfies
\begin{eqnarray}
\frac{\log\big(\pi(B(x,r))\big)}{\log r}\leq t.\label{EQNODOUBLING}
\end{eqnarray}
We set $\rho_0=\rho/2$ and $\alpha_0=\min(\rho_0,\alpha)$. Let $r\in(0,\alpha_0)$, let $i\in\{1,\dots,N\}$ be as above and let $x\in K$ with 
$B(x,r)\cap B(y_i,\rho_0)\neq\varnothing$. Then $x\in B(y_i,\rho)$ so that (\ref{EQNODOUBLING}) holds true. Thus, since $t>\Dsuppiunif(-\infty)$ is arbitrary,
$$\inf_N\inf_{\substack{y_1,\dots,y_N\in K\\\rho>0}}\limsup_{r\to 0}\inf_{i=1,\dots,N}\frac{\log\big(\inf_{B(x,r)\cap B(y_i,\rho)\neq\varnothing}\pi(B(x,r))\big)}{\log r}\leq
\Dsuppiunif(-\infty).$$
The opposite inequality is trivial, and the proof of the second assertion follows exactly the same lines.
\end{proof}

\subsection{Proof of Theorem \ref{THMMAINSMALL}, Part 1}
In this subsection, we shall prove that a generic measure $\mu\in\pk$ satisfies 
$$\dinfbigpiq(\mu)\geq\begin{cases}
-q\Dsuppiunif(-\infty)&\textrm{provided }q\geq 0\\
-q\Dinfpiunif(+\infty)&\textrm{provided }q\leq 0.
\end{cases}
$$
Firstly, let $t>\Dsuppiunif(-\infty)$ and let us prove that a generic $\mu\in\pk$ satisfies $\dinfbigpiq(\mu)\geq -qt$ for any $q\geq 0$.
Let $N\geq 1$, $y_1,\dots,y_N\in K$ and $\rho>0$ be such that 
$$\limsup_{r\to 0}\inf_{i=1,\dots,N}\frac{\log\big(\inf_{B(x,r)\cap B(y_i,\rho)\neq\varnothing}\pi(B(x,r))\big)}{\log r}<t.$$
We set $\mathcal U=\bigcap_{i=1}^N \big\{\mu\in\pk;\ \mu\big(B(y_i,\rho)\big)>0\big\}.$ $\mathcal U$ is a dense and open subset 
of $\pk$ and let us pick $\mu\in\mathcal U$. There exists $\veps>0$ such that $\mu(E)>1-\veps$ implies
$\mu\big(E\cap B(y_i,\rho)\big)>0$ for any $i=1,\dots,N$. Let now $E\subset K$ with $\mu(E)>1-\veps$ and let $r$
be sufficiently small. There exists $i\in\{1,\dots,N\}$ such that 
$$\frac{\log\big(\inf_{B(x,r)\cap B(y_i,\rho)\neq\varnothing}\pi(B(x,r))\big)}{\log r}<t.$$
Now,
\begin{eqnarray*}
\log\npiq(E,r)&\geq&\log \npiq\big(E\cap B(y_i,\rho),r\big)\\
&\geq&\log\left(\inf_{B(x,r)\cap B(y_i,\rho)\neq\varnothing}\pi\big(B(x,r)\big)^q\right)\\
&\geq&qt\log r.
\end{eqnarray*}
Hence, $\dinfpiq(E)\geq -qt$, which yields $\dinfbigpiq(\mu)\geq -qt$.

\smallskip

The proof for $q<0$ is similar, but now we have to take $t<\Dinfpiunif(+\infty)$. As before, there exist $y_1,\dots,y_N\in K$, $\rho>0$ and $\alpha>0$ such that, 
for any $r\in(0,\alpha)$, there exists $i\in\{1,\dots,N\}$ with
$$\frac{\log\big(\sup_{B(x,r)\cap B(y_i,\rho)\neq\varnothing}\pi(B(x,r))\big)}{\log r}>t.$$
We then carry on mutatis mutandis the same proof, except that now
$$\log\npiq(E,r) \geq q\log\left(\sup_{B(x,r)\cap B(y_i,\rho)\neq\varnothing}\pi\big(B(x,r)\big)\right).$$

\subsection{Proof of Theorem \ref{THMMAINSMALL}, Part 2}
In this subsection, we shall prove that a generic measure $\mu\in\pk$ satisfies 
$$\dinfbigpiq(\mu)\leq\begin{cases}
-q\Dsuppiunif(-\infty)&\textrm{provided }q\geq 0\\
-q\Dinfpiunif(+\infty)&\textrm{provided }q\leq 0.
\end{cases}
$$
We just consider the case $q\geq 0$ and let $t<\Dsuppiunif(-\infty)$. Let also $(y_n)$ be a dense sequence in $K$, let $(\rho_n)$ be a sequence decreasing to zero,
and let $(\veps_n)$ be a sequence of positive real numbers with $\sum_n\veps_n<1$. By assumption, for any $n\geq 1$, 
we may find $r_n\in(0,n^{-n})$ and points $x_1^n,\dots,x_n^n$ with $x_i^n\in B(y_i,\rho_n)$ such that, for any $i=1,\dots,n$,
$$\log\big(\pi(B(x_i^n,r_n))\big)\leq t\log r_n.$$
We set $$\Lambda_n=\left\{\sum_{i=1}^n p_i\delta_{x_i^n};\ \sum_i p_i=1,\ p_i>0\right\}$$
so that $\bigcup_{n\geq m}\Lambda_n$ is dense in $\pk$ for any $m\geq 1$. We also set $E_n=\big\{x_1^n,\dots,x_n^n\}$ so that
$\mu(E_n)=1$ for any $\mu\in\Lambda_n$. Lemma \ref{LEMTOPO1} gives us a real number $\eta_n>0$ such that
$$\forall \mu\in \Lambda_n,\ L(\mu,\nu)<\eta_n\implies \nu\big(E_n(r_n)\big)>1-\veps_n.$$
We let $F_n=E_n(r_n)$ and we consider the dense $G_\delta$-set
$$\mathcal R=\bigcap_{m\geq 1}\bigcup_{n\geq m}\bigcup_{\mu\in\Lambda_n}B_L(\mu,\eta_n).$$
Pick $\nu\in\mathcal R$. There exists a sequence $(n_k)$ going to $+\infty$ and a sequence $(\mu_{n_k})$ with 
$L(\nu,\mu_{n_k})<\eta_{n_k}$ for any $k$. Hence, $\nu\big(F_{n_k}\big)>1-\veps_{n_k}.$ We define
$G_l=\bigcap_{k\geq l}F_{n_k}$
so that $\nu(G_l)\to 1$ as $l\to+\infty$. On the other hand, for any $k\geq l$, $$G_l\subset F_{n_k}\subset \bigcup_{i=1}^{n_k}B(x_i^{n_k},r_{n_k}).$$
Using this covering of $G_l$, we get
\begin{eqnarray*}
\log\npiq(G_l,r_{n_k})&\leq&\sum_{i=1}^{n_k}\pi\big(B(x_i^{n_k},r_{n_k})\big)^q\\
&\leq&n_k r_{n_k}^{qt}.
\end{eqnarray*}
Taking the logarithm and then the liminf, this yields 
$$\dinfpiq(G_l)\leq -tq.$$
Since $\nu(G_l)$ can be arbitrarily close to 1, this implies $\dinfbigpiq(\nu)\leq -qt.$

\subsection{Proof of Theorem \ref{THMMAINSMALLBIS}, Part 1}
We turn to the study of the small lower multifractal  dimensions of a generic measure. More specifically, in this subsection,
we prove that a generic $\mu\in\pk$ satisfies $\dinfsmallpiq(\mu)\geq -q\Dsuppimax(-\infty)$ for any $q\geq 0$. 
Hence, let $t>\Dsuppimax(-\infty)$. Let $(y_n)_n$ be a dense sequence in $K$ and let $(\rho_n)_n$ be a sequence of positive
real numbers decreasing to zero. Let us fix $n\geq 1$. One may find $\alpha_n>0$ such that, for any $r\in(0,\alpha_n)$,
for any $k\in\{1,\dots,n\}$, for any $x\in K$ such that $B(x,r)\cap B(y_k,\rho_n)\neq\varnothing$, 
$$\log\pi\big(B(x,r)\big)\geq t\log r.$$
We then set 
\begin{eqnarray*}
\Lambda_n&=&\left\{\sum_{i=1}^n p_i\delta_{y_i};\ p_i>0,\ \sum_i p_i=1\right\}\\
F_n&=&\{y_1,\dots,y_n\}.
\end{eqnarray*}
Any $\mu\in\Lambda_n$ satisfies $\mu(F_n)=1$. Hence, we may find $\eta_n>0$ such that $\nu\big(F_n(\rho_n)\big)>1-1/n$ provided $L(\mu,\nu)<\eta_n$. We finally
consider
$$\mathcal R=\bigcap_{m\geq 1}\bigcup_{n\geq m}\bigcup_{\mu\in\Lambda_n}B_L(\mu,\eta_n).$$
Pick $\nu$ in the dense $G_\delta$-set $\mathcal R$ and let $E\subset K$ with $\nu(E)>0$. We may find $n$ as large as we want such that
$\nu\big(E\cap F_n(\rho_n)\big)>0$. Now, for any $r\in(0,\alpha_n)$, 
\begin{eqnarray*}
 \log \npiq(E,r)&\geq&\log\npiq\big(E\cap F_n(\rho_n),r)\\
&\geq&\log\left(\inf_{B(x,r)\cap F_n(\rho_n)\neq\varnothing}\pi\big(B(x,r)\big)^q\right)\\
&\geq&qt\log r.
\end{eqnarray*}
Hence, $\dinfsmallpiq(\nu)\geq -qt$.

\subsection{Proof of Theorem \ref{THMMAINSMALLBIS}, Part 2}
We conclude the proof of Theorem \ref{THMMAINSMALLBIS} by showing that a generic $\mu\in\pk$ satisfies $\dinfsmallpiq(\mu)\leq -q\Dsuppiunifmax(-\infty)$
for any $q\geq 0$. We begin by fixing $t<\Dsuppiunifmax(-\infty).$ There exists $z\in K$ and $\kappa>0$ such that
$$t<\inf_{\substack{y_1,\dots,y_N\in B(z,\kappa)\\\rho>0}}\limsup_{r\to 0}\inf_{i=1,\dots,N}\frac{\log\big(\inf_{x\in B(y_i,\rho)}\pi(B(x,r))\big)}{\log r}.$$
 The proof now follows that of Part 2 of Theorem \ref{THMMAINSMALL}, except that we "localize" it in $K\cap B(z,\kappa)$. Specifically,
we now consider $(y_n)$ a dense sequence in $K\cap B(z,\kappa)$. We construct the sequence $(\rho_n)$, $(\veps_n)$, 
$(r_n)$ and $(x_n^i)$ as above, but starting from this sequence $(y_n)$ and from the property 
$$\forall n\geq 1,\ \limsup_{r\to 0}\inf_{i=1,\dots,n}\frac{\log\big(\inf_{x\in B(y_i,\rho_n)}\pi(B(x,r))\big)}{\log r}\geq t.$$
We also ask that for any $n\geq 1$ and any $i\in\{1,\dots,n\}$, $B(x_i^n,r_n)$ is contained in $B(z,\kappa)$. Next, for any $n\geq 1$, we now set
\begin{eqnarray*}
\Lambda_n&=&\Bigg\{\lambda\sum_{i=1}^n p_i\delta_{x_i^n}+(1-\lambda)\theta;\ \lambda,p_i\in(0,1),\ \sum_ip_i=1, \theta\in\pk,\\
&&\quad\quad\ \supp(\theta)\cap B(z,\kappa+2r_n)=\varnothing\Bigg\}\\
E_n&=&\left\{x_1^n,\dots,x_n^n\right\}\\
F_n&=&E_n(r_n).
\end{eqnarray*}
It is not hard to show that, for any $m\geq 1$, $\bigcup_{n\geq m}\Lambda_n$ keeps dense in $\pk$. Moreover, for any $\mu\in \Lambda_n$, we may
find $\eta_{n,\mu}>0$ such that
$$L(\nu,\mu)<\eta_{n,\mu}\implies
\left\{\begin{array}{rcl}
\nu(F_n)&\geq&\lambda(1-\veps_n)\\
\nu\big(B(z,\kappa)\big)&\leq&\lambda(1-\veps_n)^{-1}.
\end{array}\right.$$
Let $\mathcal R$ be the dense $G_\delta$-subset of $\pk$ defined by
$$\mathcal R=\bigcap_{m\geq 1}\bigcup_{n\geq m}\bigcup_{\mu\in\Lambda_n}B_L(\mu,\delta_{n,\mu})\cap\big\{\nu\in\pk;\ \nu\big(B(z,\kappa)\big)>0\big\}.$$
Let $\nu\in\mathcal R$ and let $(n_k)$ be a sequence growing to $+\infty$ such that 
$$\nu(F_{n_k})\geq (1-\veps_{n_k})^2\nu\big(B(z,\kappa)\big)$$
for any $k\geq 1$. We finally define $G=\bigcap_n F_{n_k}$. Since any $F_n$ is contained in $B(z,\kappa)$, the previous inequality
ensures that $\nu(G)>0$ provided $(\veps_n)$ goes sufficiently fast to 0. On the other hand, for any $k\geq 1$, 
$$G\subset F_{n_k}\subset \bigcup_{i=1}^{n_k}B(x_i^{n_k},r_{n_k}).$$
This yields (see Part 2 of Theorem \ref{THMMAINSMALL})
$$\npiq(G,r_{n_k})\leq n_k r_{n_k}^{qt}$$
so that $\dinfsmallpiq(\nu)\leq -qt$.

\subsection{Application to self-similar sets}

We now show how to apply Theorems \ref{THMMAINSMALL} and \ref{THMMAINSMALLBIS} to self-similar compact sets. Let $M\geq 2$, let $S_1,\dots,S_M:\mathbb R^d\to\mathbb R^d$
be contracting similarities with respective ratio $r_1,\dots,r_M\in(0,1)$. Let $(p_1,\dots,p_M)$ be a probability vector. Let $K$ be the nonempty compact
subset of $\mathbb R^d$ and let $\pi$ be the probability measure in $\pk$ satisfying
\begin{eqnarray*}
K&=&\bigcup_{m=1}^M S_i(K)\\
\pi&=&\sum_{m=1}^M p_i \pi\circ S_m^{-1}.
\end{eqnarray*}
We just need to prove the following proposition.
\begin{proposition}
Let $K$ and $\pi$ be as above and assume that the Open Set Condition is satisfied. Define 
$$s_{\min}=\min_m \frac{\log p_m}{\log r_m}\textrm{ and }s_{\max}=\max_m \frac{\log p_m}{\log r_m}.$$
Then 
$$\begin{array}{rcccccl}
\Dsuppiunif(-\infty)&=&\Dsuppiunifmax(-\infty)&=&\Dsuppimax(-\infty)&=&s_{\max}\\
\Dinfpiunif(+\infty)&=&\Dinfpiunifmin(+\infty)&=&\Dinfpimin(-\infty)&=&s_{\min}.
\end{array}$$
\end{proposition}
\begin{proof}
We just give the proof of the first inequality. It is straightforward to check that 
$$\Dsuppimax(-\infty)\geq \Dsuppiunifmax(-\infty)\geq\Dsuppiunif(-\infty).$$
Thus we just need to prove that
$$\Dsuppiunif(-\infty)\geq s_{\max}\textrm{ and }\Dsuppimax(-\infty)\leq s_{\max}.$$
Without loss of generality, we may assume that the diameter of $K$ is less than 1.
We shall use standard notations which can be found e.g. in \cite{Fal97}. For a word $\mathbf m=(m_1,\dots,m_n)$
in $\{1,\dots,M\}^n$ of length $n$, let 
\begin{eqnarray*}
S_{\mathbf m}&=&S_{m_1}\circ\dots\circ S_{m_n}\\
p_{\mathbf m}&=&p_{m_1}\times\dots\times p_{m_n}\\
r_{\mathbf m}&=&r_{m_1}\times\dots\times r_{m_n}.
\end{eqnarray*}
If the word $\mathbf m$ is infinite, then $S_{\mathbf m}(K)=\bigcap_{i=1}^{+\infty}S_{m_i}(K)$ is reduced to a single point $x_{\mathbf m}\in K$
and each point of $K$ is uniquely defined by such a word. Let now $y\in K$, $\rho>0$ and let $l$ be such that $\frac{\log p_l}{\log r_l}=s_{\max}$.
There exists a word $\mathbf m=(m_1,\dots,m_n)$ such that $S_{\mathbf m}(K)\subset B(y,\rho)$. We then define
$$\overline{\mathbf m}=(m_1,\dots,m_n,l,\dots)$$
$$\overline{\mathbf m_k}=(m_1,\dots,m_n,l,\dots,l)$$
where $l$ appears $k$ times at the end of $\overline{\mathbf m_k}$. We define $x_y$ as $S_{\overline{\mathbf m}}(K)$. Now, for any $k\geq 1$, there exists $z\in K$ such that 
$x=S_{\overline{\mathbf m_k}}z$, so that $B(x_y,r_{\overline{ \mathbf m_k}})=S_{\overline{\mathbf m_k}}(B(z,1))$. 
Now the definition of $\pi$ and the open set condition ensure that
$$\pi\big(S_{\overline{\mathbf m_k}}(B(z,1))\big)=p_{\overline{\mathbf m_k}}\pi\big(B(z,1)\big)=p_{\overline{\mathbf m_k}}$$
since the diameter of $K$ is less than 1. Thus, for any $k\geq 1$,
\begin{eqnarray*}
\pi\big(B(x_y,r_l^{k+n})\big)&\leq &\pi\big(B(x_y,r_{\overline{ \mathbf m_k}})\big)\\
&\leq&p_{\overline{\mathbf m_k}}=p_{m_1}\dots p_{m_n}p_l^k.
\end{eqnarray*}
Finally, let $N\geq 1$, let $y_1,\dots,y_N\in K$ and let $\rho>0$. To each $y_i$, we can associate a word $\mathbf m^i$ of length $n^i$ and a point $x_i$ as above. 
Let $n=\max(n^i)$. Then for any $i=1,\dots,N$,
$$\frac{\log \pi\big(B(x_i,r_l^{k+n})\big)}{(k+n)\log r_l}\geq \frac C{k+n}+\frac{k}{k+n}s_{\max}$$
where $C$ does not depend on $k$. Letting $k$ to $+\infty$ gives $\Dsuppiunif(-\infty)\geq s_{\max}.$
On the other hand, it is well known that $\Dsuppi(-\infty)\leq s_{\max}$ (see for instance \cite{Pat97}). By the homogeneity
of self-similar sets and self-similar measures, this implies $\Dsuppimax(-\infty)\leq s_{\max}$.
\end{proof}

\bibliographystyle{amsalpha}

\providecommand{\bysame}{\leavevmode\hbox to3em{\hrulefill}\thinspace}
\providecommand{\MR}{\relax\ifhmode\unskip\space\fi MR }
\providecommand{\MRhref}[2]{%
  \href{http://www.ams.org/mathscinet-getitem?mr=#1}{#2}
}
\providecommand{\href}[2]{#2}
\begin{thebibliography}{}

\end{thebibliography}


\begin{thebibliography}{Bay12}

\bibitem[Bay12]{BAYLQ}
F.~Bayart, \emph{{How behave the typical $L^q$-dimensions of measures?}},
  preprint (2012), arXiv:1203.2813.

\bibitem[Fal97]{Fal97}
K.~Falconer, \emph{{Techniques in Fractal Geometry}}, Wiley, 1997.

\bibitem[MR02]{MR02}
J.~Myjak and R.~Rudnicki, \emph{{On the box dimension of typical measures}},
  Monat. Math. \textbf{136} (2002), 143--150.

\bibitem[Ols05]{OL05}
L.~Olsen, \emph{{Typical $L^q$-dimensions of measures}}, Monatsh. Math.
  \textbf{146} (2005), 143--157.

\bibitem[Ols11]{OL11}
\bysame, \emph{Typical multifractal box dimensions of measures}, Fund. Math.
  \textbf{211} (2011), 245--266.

\bibitem[Par67]{Par67}
K.R. Parthasarathy, \emph{Probability measures on metric spaces}, Probability
  and Mathematical Statistics, Academic Press, 1967.

\bibitem[Pat97]{Pat97}
N.~Patzschke, \emph{Self-conformal multifractal measures}, Adv. Appl. Math.
  \textbf{19} (1997), 486--513.

\end{thebibliography}
\providecommand{\bysame}{\leavevmode\hbox to3em{\hrulefill}\thinspace}
\providecommand{\MR}{\relax\ifhmode\unskip\space\fi MR }
\providecommand{\MRhref}[2]{%
  \href{http://www.ams.org/mathscinet-getitem?mr=#1}{#2}
}
\providecommand{\href}[2]{#2}

\end{document}